\documentclass[1p]{elsarticle}

\usepackage{auto-pst-pdf}
\usepackage{graphicx}
\usepackage{psfrag}
\usepackage{fancybox}
\usepackage[centertags]{amsmath}
\usepackage{amsfonts,amssymb,amsthm}
\usepackage{float,longtable}
\usepackage[centertags]{amsmath}
\usepackage{amsfonts,amssymb,amsthm}
\usepackage{float,longtable}
\usepackage{graphicx}
\usepackage{a4wide}
\usepackage[latin1]{inputenc}
\usepackage{newlfont}
\usepackage{pifont}

\usepackage{geometry}

\newtheorem{Theorem}{Theorem} 
\newtheorem{Proposition}[Theorem]{Proposition}
\newtheorem{Remark}[Theorem]{Remark}
\newtheorem{Lemma}[Theorem]{Lemma}

\numberwithin{Theorem}{section}
\numberwithin{equation}{section}

\newcommand{\ep}{\varepsilon}

\title{Convergence of a {semi-discretization} scheme for the Hamilton--Jacobi equation:
a new approach with the adjoint method}

\author[ist]{F.~Cagnetti}
\ead{cagnetti@math.ist.utl.pt}

\author[istk]{D.~Gomes\corref{cor1}}
\ead{dgomes@math.ist.utl.pt}

\author[ber]{H.V.~Tran}
\ead{tvhung@math.berkeley.edu}

\cortext[cor1]{Corresponding author}

\address[ist]{Departamento de Matem\'atica, Instituto Superior T\'ecnico,
 Av. Rovisco Pais, 1049-001, Lisbon}

\address[istk]{Departamento de Matem\'atica, Instituto Superior T\'ecnico,
 Av. Rovisco Pais, 1049-001, Lisbon
and
King Abdullah University of Science and Technology (KAUST), CSMSE Division , Thuwal 23955-6900. Saudi Arabia}

\address[ber]{Department of Mathematics,
University of California Berkeley, CA, 94720-3840}

\begin{document}

\begin{abstract}
We consider 
a numerical scheme for the one dimensional time dependent Hamilton-Jacobi equation in the periodic setting.
This scheme consists in a semi-discretization using monotone approximations of the Hamiltonian in the spacial variable. 
From classical viscosity solution theory, these schemes are known to converge. In this paper
we present a new approach to the study 
of the rate of convergence of the approximations
based on the nonlinear adjoint method recently introduced by L. C. Evans.
We estimate the rate of convergence for convex Hamiltonians
and recover the $O(\sqrt h)$ convergence rate in terms of the $L^\infty$ norm 
and $O(h)$ in terms of the $L^1$ norm, 
 where $h$ is the size of the spacial 
 grid. We discuss also possible generalizations to higher dimensional problems and present 
several other additional estimates. The special case of quadratic Hamiltonians is considered in 
detail in the end of the paper. 
\end{abstract}


\maketitle

\author[ist]{F.~Cagnetti\corref{cor1}}
\ead{cagnetti@math.ist.utl.pt}

\author[ist]{D.~Gomes\corref{cor2}}
\ead{dgomes@math.ist.utl.pt}

\author[ber]{H.V.~Tran\corref{cor3}}
\ead{tvhung@math.berkeley.edu}

%%%%%%%%%%%%%%%%%%%%%%%%%%%%%%%%%%%%%%%%%%%%%%%

\begin{section}{Introduction}

{We consider in this paper a semi-discretization 
of the one dimensional time dependent Hamilton--Jacobi equation in the periodic setting:
\begin{equation} \label{Ham}
\left\{ \begin{aligned}
u_t + H (u_x) &= 0, \quad \quad \, \ \text{ in } \mathbb{T} \times (0,\infty), \vspace{.05in} \\
u &= u_0 , \quad \quad \text{ on } \mathbb{T} \times \{ t = 0 \},
\end{aligned} \right. 
\end{equation}
providing approximations and error estimates for the viscosity solutions.}

{As for the Hamiltonian $H: \mathbb{R} \to \mathbb{R}$, we assume
\begin{itemize}
\item[(H$1$)] $H$ smooth and convex;
\item[(H$2$)] $H$ coercive. i.e. $\lim_{|p| \to \infty} H ( p ) = + \infty$.
\end{itemize}}
Moreover, $u_0 : \mathbb{T} \to \mathbb{R}$ is a given smooth function, and
$\mathbb{T}$ is the one dimensional torus identified, when convenient, with the interval $[0,1]$. 
Several authors investigated equation \eqref{Ham} and related problems, and a number of results are available
in literature (see \cite{CL, Sou, BS91, BCD, DJ98, K00, BarJak1, FalFer, BarJak2, J06, AO06, JKC08,CCDG, AO10, Abg,FF, CV, FalFer2},  to name just a few). 

The aim of this note is to take a first step on a new approach to this problem, 
using the adjoint method recently introduced by Evans  (see \cite{E2}, and also 
\cite{T1, CGT1, E3, CGT2}).
Indeed, we will show how it is possible to recover some results, which are already well-known in literature,
with new and easy proofs. 

{For the sake of simplicity we consider only the one dimensional setting.
Nevertheless, most of the results can be extended
without major changes to higher dimensions, with the exception of Section \ref{42}, 
where the argument we use is indeed one dimensional (See Section \ref{generalization} for details).}
%To focus on the main ideas of our approach, we try to keep the formulation as simple as possible, 
%while a more detailed study will be the subject of a forthcoming paper.

We consider a function 
$F: \mathbb{R} \times \mathbb{R} \to \mathbb{R}$
with the following properties:
\begin{itemize}

\item[(F$1$)] $F$ is convex;

\item[(F$2$)] $F (\cdot, q)$ is increasing for each $q \in \mathbb{R}$
and $F (p, \cdot)$ is increasing for each $p \in \mathbb{R}$;

\item[(F$3$)] $F(-p,p) = H(p)$ for every $p \in \mathbb{R}$.

\end{itemize}
We call $F$ a {\it numerical Hamiltonian} of the {semi-discrete} scheme.
Such a function appears naturally.
Indeed, if for instance $H(0)=0=\min_{p\in \mathbb{R}} H(p)$, 
then $F$ can be chosen as follows. Setting
$$
F_1(p):=
\begin{cases}
0 \qquad &p \le 0,\\
H(-p) & p>0,
\end{cases}
\hspace{1cm}
%\text{ and }
\hspace{1cm}
F_2(q):=
\begin{cases}
0 \qquad &q \le 0,\\
H(q) & q>0,
\end{cases}
$$
and $F(p,q) :=F_1(p)+F_2(q)$ for $(p,q)\in \mathbb{R}^2$, 
properties (F1)--(F3) are satisfied.
{Other possible choices of $F$ will be mentioned below.}

At this point, for every $h > 0$ we introduce the solution $u^h: \mathbb{T} \times [0,\infty) \to \mathbb{R}$ to:
\begin{equation} \label{y}
\left\{ \begin{aligned}
u^h_t + F \left( - \delta_h u^h  , \delta_{-h} u^h \right) &= 0, 
\quad \quad \, \, \, \text{ in } \mathbb{T} \times (0,\infty), \vspace{.05in} \\
u &= u_0 , \quad \quad \text{ on } \mathbb{T} \times \{ t = 0 \},
\end{aligned} \right. 
\end{equation}
where for every function $v: \mathbb{T} \to \mathbb{R}$ we set
\begin{equation*}
 \delta_h v (x) := \frac{v^h (x + h) - v^h (x)}{h}, \quad \quad x \in \mathbb{T}.
\end{equation*}
Existence and uniqueness of $u^h$ can be easily proven (see the Appendix).
%
%We point out that \eqref{y} is not a standard approximation for equation \eqref{Ham}, 
%for several reasons.
%First, we are not discretizing in the time variable.
%Also, note that the function $u^h$ is defined in \textit{all} the torus $\mathbb{T}$, 
%and not only in grid points.
%Finally, $h$ can take any value in $\mathbb{R}$.
%This gives us the advantage that we can differentiate $u^h$ with respect to
%the grid size without any technical problem.  

{
Let us notice that $h$ can take any value in $(0,\infty)$, which makes it possible to consider 
the derivate of $u^h$ with respect to the grid size.
}

We state now our main results.
The first one concerns the $L^\infty$-error estimate for the approximate solutions.
\begin{Theorem} \label{main}
{Let $F$ satisfy (F1)--(F3), and let $u^h$ solve \eqref{y}}. Then, 
for every $T \in (0, \infty)$ there exists a positive constant $C = C (T)$, 
independent of $h$, such that
\begin{equation} \label{est}
\sup_{ t \in [0, T] } \| u (\cdot, t) - u^h (\cdot, t)\|_{L^{\infty} (\mathbb{T})} \leq C \sqrt{h},
\end{equation}
{
where $u$ is the unique viscosity solution of \eqref{Ham}}.
\end{Theorem}
As already mentioned, inequality \eqref{est} is not new in literature
and appeared, for instance, in the seminal paper \cite{CL},
where Crandall and Lions studied Hamilton--Jacobi equation for coercive 
(not necessarily convex) Hamiltonians.

{
Another possible choice for the numerical Hamiltonian is}
\begin{equation}\label{H.CL}
F(p,q)=H \left( \dfrac{q-p}{2} \right)+\gamma(p+q),
\end{equation}
where $\gamma$ is a positive constant chosen in such a way
that $|H'(p)| \leq 2\gamma$ for $|p| \le R$, with
$R > 0$ playing the role of an a priori bound on $|u_x|$.
Note that, under this assumption, conditions (F2)--(F3) are satisfied,
and \eqref{y} reads as
\begin{equation} \label{gs}
u_t^h+H \left( \dfrac{u^h(x+h,t)-u^h(x-h,t)}{2h} \right) = \gamma h \Delta_h u^h,
\end{equation}
where for every function $v: \mathbb{T} \to \mathbb{R}$ we set
\begin{equation*}
\Delta_h v (x):= \frac{v (x + h) - 2 v (x) + v (x - h)}{h^2},
\quad \quad x \in \mathbb{T}.
\end{equation*}
Equation \eqref{gs} is the analog to the usual regularized Hamilton--Jacobi equation
$u^\varepsilon_t + H (Du^\varepsilon) = \varepsilon \Delta u^\varepsilon$
 (see also Crandall and Majda \cite{CraMaj}, and Souganidis \cite{Sou}), 
with the additional {feature} that the viscosity term vanishes as the grid size goes to zero.

{
Next theorem provides an $L^1$-error estimate for the approximate solutions,
when the numerical Hamiltonian is of the form \eqref{H.CL}.
\begin{Theorem}\label{main2}
{Let $F$ be given by \eqref{H.CL}, and let $u^h$ solve \eqref{y}.
Then,} for every $T \in (0,\infty)$ there exists a positive constant $C=C(T)$,
independent of $h$, such that
$$
\|u^h(\cdot,t)-u(\cdot,t)\|_{L^1(\mathbb T)} \leq Ch, \quad \text{for } t \in [0,T],
$$
where $u$ is the unique viscosity solution of \eqref{Ham}.
\end{Theorem}
Lin and Tadmor \cite{LiTa} derived a version of Theorem \ref{main2} by using a method essentially related to
the Adjoint Method. See Theorem 2.1 in \cite{LiTa} for details.
}

Let us now briefly comment on the main ingredient of the present paper, 
that is how we prove Theorems \ref{main}, \ref{main2}.
We start by linearizing \eqref{y}, and then we consider the adjoint of the equation obtained, 
with various terminal data (see \eqref{js}, \eqref{L1ad}).
Using properties of the solutions of the adjoint equations and integration by parts techniques,
we are able to prove the necessary estimates.
{In particular, we show that the sequence $\{ u^h \}_{h \in \mathbb{N}}$ converges uniformly,
and this, by the properties of viscosity solutions, implies that the limit of the sequence is the solution $u$
of \eqref{Ham}.}

It is extremely interesting that both {the} $L^\infty$ and $L^1$ error estimates can be treated
in the same way by using the Adjoint Method in a direct way.

We conclude by observing that, for technical reasons, at the moment
we are not able to remove the convexity assumption on $H$ in Theorem \ref{main} (see Remark \ref{convex1}).

The paper is organized as follows.
Section \ref{2} contains some preliminary observations, concerning finite difference quotients.
Section \ref{4} is devoted to the proofs of Theorems \ref{main}, \ref{main2}
and to their generalizations to higher dimensional spaces. 
%In Section \ref{3} we face the linear case, 
%while a general convex Hamiltonian is the object of Section \ref{4}.
%The special case $H(p)= p^2/2$ is considered in Section \ref{special.case}.
Finally, details about existence, uniqueness, and smoothness 
of the solution $u^h$ of \eqref{y} are given in the Appendix. 

We would like to thank Roberto Ferretti for bringing
to our attention the interesting paper of Lin and Tadmor \cite{LiTa}.
We also thank the anonymous referees for useful comments and suggestions.

F. Cagnetti was supported by the UTAustin$|$Portugal partnership through the FCT post-doctoral fellowship
SFRH / BPD / 51349 / 2011, CAMGSD-LARSys through FCT Program POCTI -FEDER and by grants PTDC/MAT/114397/2009,
UTAustin/MAT/0057/2008, and UTA-CMU/MAT/0007/2009.
D. Gomes was partially supported by CAMGSD-LARSys through FCT Program POCTI - FEDER and by grants PTDC/MAT/114397/2009,
UTAustin/MAT/0057/2008, and UTA-CMU/MAT/0007/2009.
H. Tran was supported in part by VEF fellowship.

\end{section}

%%%%%%%%%%%%%%%%%%%%%%%%%%%%%%%%%%%%%%%%%%%%%%%

\begin{section}{A few facts about finite difference quotients} \label{2}

For the convenience of the reader, 
we recall in this section a few facts about  calculus with finite differences,
whose proofs are elementary.
\begin{Lemma}
Let $u,w: \mathbb{T} \to \mathbb{R}$, and let $h \in \mathbb{R}$.
Then, for every $x \in \mathbb{T}$
\begin{eqnarray} 
& \delta_h w (x - h) = \delta_{-h} w (x); \label{yh} \\
& \delta_{-h} \left[ \delta_h w \right] (x) 
= \delta_{h} \left[ \delta_{-h} w \right] (x) 
= \Delta_h w (x), \label{ik} \\
&\delta^2_h w (x) = \Delta_h w (x+h), \label{yb} \\
& \left[ \delta_h (v w) \right] (x) = v (x+ h) \, \delta_h w (x) + w (x) \, \delta_h v (x) \label{vw} \\
&\delta_h \left[ w^2 (x) \right] = 2 w (x) \delta_h w (x) + h \left[ \delta_h w (x) \right]^2 \label{ww}\\
& \Delta_h [w^2(x)]=2 w(x) \Delta_h w(x)+ (\delta_h w(x))^2 +(\delta_{-h} w(x))^2 \label{w2}
\end{eqnarray}
\end{Lemma}
The following lemma gives a discrete version of integration by parts.
\begin{Lemma} \label{byparts}
Let $v, w \in L^2(\mathbb{T})$ and let $h \in \mathbb{R}$. Then
\begin{equation*}
 \int_{\mathbb{T}} w \, \delta_h v \, dx
= - \int_{\mathbb{T}} v \, \delta_{-h} w \, dx.
\end{equation*}
\end{Lemma}
We also recall the following formula
\begin{Lemma}
Let $v, w \in L^2(\mathbb{T})$ and let $h \in \mathbb{R}$. Then
\begin{equation*}
 \int_{\mathbb{T}} \delta_h v \, \delta_h w \, dx
= - \int_{\mathbb{T}} w \, \Delta_h v \, dx.
\end{equation*}
\end{Lemma}
\end{section}

\begin{section}{Adjoint Method and Error Estimates} \label{4}

For every $h > 0$, we consider the following equation:
\begin{equation}\label{HJ.approx}
\left\{ \begin{aligned}
u^h_t + F \left( - \delta_h u^h  , \delta_{-h} u^h \right) &= 0, 
\quad \quad \, \, \, \text{ in } \mathbb{T} \times (0,\infty), \vspace{.05in} \\
u^h &= u_0 , \quad \quad \text{ on } \mathbb{T} \times \{ t = 0 \}.
\end{aligned} \right. 
\end{equation}
Next proposition, whose proof can be found in the Appendix, 
shows that existence and uniqueness of the smooth solution 
of the above equation are guaranteed.
\begin{Proposition} \label{ODE1}
Let $h > 0$, and assume that $F \in C^2 (\mathbb{R}^2)$ and $u_0 \in C^2 (\mathbb{T})$. 
Then, there exists a unique solution $u^h$ to \eqref{HJ.approx}.
Moreover, we have $u^h,  u^h_x, u^h_{xx} \in C (\mathbb{T} \times [0, \infty))$ and
\begin{equation*}
u^h (x,\cdot), u^h_x (x,\cdot), u^h_{xx} (x,\cdot) \in C^1 ( [ 0, \infty) ) \quad \text{ for every }x \in \mathbb{T}.
\end{equation*} 
\end{Proposition}
{We can now begin the proof of our main results.
In this section, all the first and second derivatives of $F$ 
will be evaluated at $( - \delta_h u^h , \delta_{-h} u^h )$.}

%%%%%%%%%%%%%%%%%%%%%%%%%%%%%%%%%%%%%%

\subsection{$L^\infty$-error estimates} \label{Linf}

We now introduce the Adjoint Method and use it to prove Theorem \ref{main}.
We consider the formal linearized operator $L^h$ 
corresponding to equation \eqref{HJ.approx}:
\begin{equation} \label{y2}
v \mapsto L^h v = v_t - D_p F (\delta_h v) + D_q F (\delta_{-h} v) ,
\end{equation}
where $D_p F$ and $D_q F$ are evaluated at $( - \delta_h u^h , \delta_{-h} u^h )$.
For each $h > 0$, $x_0 \in \mathbb{T}$ and $T \in (0, \infty)$
we denote by $\sigma^{h,x_0,T}$ the solution to
\begin{equation} \label{js} 
\left\{ \begin{aligned}
- \sigma^{h,x_0,T}_t + \delta_{-h}( \sigma^{h,x_0,T} D_p F ) 
- \delta_{h}( \sigma^{h,x_0,T} D_q F ) 
&= 0, \, \hspace{2cm} \text{ in } \mathbb{T} \times [0, T ), \\
\sigma^{h,x_0,T} &= \delta_{x_0}, \quad \quad \quad \quad \quad \text{ on } \mathbb{T} \times \{ t= T \},
\end{aligned} \right. 
\end{equation}
{where $\delta_{x_0}$ denotes the Dirac delta measure concentrated at $x_0$.}

\begin{Proposition}[Properties of $\sigma^{h,x_0,T}$] \label{km}
Let $h > 0$, $x_0 \in \mathbb{T}$, and $T > 0$.
For every $t \in [0,T]$ $\sigma^{h,x_0,T} (\cdot, t)$ is a probability measure on $\mathbb{T}$.
\end{Proposition}

\begin{proof}

Let us fix $t_2 \in (0,T)$. We will proceed by steps.

\vspace{4pt} 
\noindent \textbf{Step 1: $\sigma^{h,x_0,T} (\cdot, t_2) \geq 0$}. 

In order to show that $\sigma^{h,x_0,T} (\cdot, t_2)$ is non-negative, 
for every $f \in C^{\infty} (\mathbb{T})$ let us denote  
by $v^{h,f,t_2}$ the solution of the adjoint of the equation \eqref{js}:
\begin{equation} \label{td}
\left\{ \begin{aligned}
v^{h,f,t_2}_t - D_p F ( \delta_h v^{h,f,t_2} ) + D_q F ( \delta_{-h} v^{h,f,t_2} ) 
&= 0, \quad \quad \text{ in } \mathbb{T} \times ( t_2, \infty), \\
v^{h,f,t_2} &= f, \quad \quad \text{ on } \mathbb{T} \times \{ t= t_2 \}.
\end{aligned} \right. 
\end{equation}
First of all, observe that 
\begin{equation} \label{pos}
f \geq 0 \Longrightarrow v^{h,f,t_2} \geq 0, \quad \text{ in }\mathbb{T} \times [t_2,\infty).
\end{equation}
Indeed, let $f \geq 0$, and for every $\varepsilon > 0$
set $z^{\varepsilon}:= v^{h,f,t_2} + \varepsilon t$.
{We have}
\begin{equation*}
\min_{(x,t) \in \mathbb{T} \times [t_2,T]} v^{h,f,t_2} (x,t) + \varepsilon T 
\geq \min_{(x,t) \in \mathbb{T} \times [t_2,T]} z^{\varepsilon} (x,t)
= \min_{x \in \mathbb{T} } z^{\varepsilon} (x,t_2) 
= \min_{x \in \mathbb{T} } f (x) + \varepsilon t_2,
\end{equation*}
so that
\begin{equation*}
\min_{(x,t) \in \mathbb{T} \times [t_2,T]} v^{h,f,t_2} (x,t)
\geq \min_{x \in \mathbb{T} } f (x) - \varepsilon {(T-t_2).}
\end{equation*}
Sending $\varepsilon \to 0^+$ claim \eqref{pos} follows.

Let us now multiply equation \eqref{js} by $v^{h,f,t_2}$ and integrate, to get
\begin{equation*} 
-\int_{t_2}^{T} \int_{\mathbb{T}} v^{h,f,t_2} \, \sigma^{h,x_0,T}_t \, dx \, ds
+ \int_{t_2}^{T} \int_{\mathbb{T}} v^{h,f,t_2}
\left[ \delta_{-h}( \sigma^{h,x_0,T} D_p F ) 
- \delta_{h}( \sigma^{h,x_0,T} D_q F )  \right] \, dx \, ds = 0.
\end{equation*}
Integrating by parts the first term becomes
\begin{equation*}
- \int_{t_2}^{T} \int_{\mathbb{T}} v^{h,f,t_2} \, \sigma^{h,x_0,T}_t \, dx \, ds
= - v^{h,f,t_2} ( x_0 ,T) 
+ \int_{\mathbb{T}} f(x) \, \sigma^{h,x_0,T} (x,t_2) \, dx
+ \int_{t_2}^{T} \int_{\mathbb{T}} v^{h,f,t_2}_t \, \sigma^{h,x_0,T} \, dx \, ds.
\end{equation*}
Thanks to \eqref{pos}, combining the last two equalities, 
integrating by parts, and using equation \eqref{td}, we obtain
\begin{equation*}
\int_{\mathbb{T}} f(x) \, \sigma^{h,x_0,T} (x,t_2) \, dx 
= v^{h,f,t_2} ( x_0 ,T) \geq 0, \quad \quad \text{ for each }f \geq 0,
\end{equation*}
from which we deduce that $\sigma^{h,x_0,T} ( \cdot ,t_2) \geq 0$.

\vspace{4pt} 
\noindent \textbf{Step 2: $\sigma^{h,x_0,T} (\cdot, t_2)$ has total mass $1$}. 

We integrate 
\eqref{js} from $t_2$ to $T$ and over $\mathbb{T}$, to get
\begin{equation*}
\begin{split}
& 1 - \int_{\mathbb{T}} \sigma^{h,x_0,T} (x,t_2) \, dx  
 = \int_{t_2}^{T} \int_{\mathbb{T}} \sigma^{h,x_0,T}_t (x,s) \, dx \, d s \\
 &= \int_{t_2}^{T} \int_{\mathbb{T}} \left[  \delta_{-h}( \sigma^{h,x_0,T} D_p F ) 
- \delta_{h}( \sigma^{h,x_0,T} D_q F ) \right] \, dx \, ds = 0,
\end{split}
\end{equation*}
by periodicity.
\end{proof}
The following proposition establishes a useful formula.
\begin{Proposition} \label{formula}
Let $h > 0, x_0 \in \mathbb{T}$, and $T \in (0, + \infty)$. Then 
\begin{equation*}
\int_0^{T} \int_{\mathbb{T}}  \sigma^{h,x_0,T} L^h \theta \, dx \, dt
= \theta (x_0,T) - \int_{\mathbb{T}} \theta (x,0) \sigma^{h,x_0,T} (x,0) \, dx,
\end{equation*}
whenever $\theta \in C (\mathbb{T} \times [0, \infty))$ is such that
$\theta (x, \cdot) \in C^1 ([0, \infty))$ for every $x \in \mathbb{T}$.
\end{Proposition}

\begin{proof}
Multiplying equation \eqref{js} by $\theta$ and integrating by parts, we have
\begin{equation*}
- \left[ \int_{\mathbb{T}}  \sigma^{h,x_0,T} \theta \, dx \right]_{0}^{T} +
\int_0^{T} \int_{\mathbb{T}}  \sigma^{h,x_0,T} L^h \theta \, dx \, dt = 0, 
\end{equation*}
and this shows the identity.
\end{proof}
In the next proposition we derive some useful equations.
\begin{Proposition}
The following equations are satisfied in $\mathbb{T} \times (0, \infty)$:
\begin{equation} \label{js2} 
\left\{ \begin{aligned}
&L^h u^h_x  = 0, \\
&L^{h} u^h_{xx} + D_{pp} F (\delta_h u^h_x)^2 + D_{qq} F (\delta_{-h} u^h_x)^2 
+ 2 D_{pq} F (-\delta_h u^h_x)(\delta_{-h} u^h_x) = 0,\\
&L^{h} w + \frac{h}{2} D_p F (\delta_h u^h_x)^2 + \frac{h}{2} D_q F (\delta_{-h} u^h_x)^2 = 0, \\
& L^{h} u^h_h - \frac{1}{h} D_p F \left[  u_x^h \mid_ {x+h} - \delta_h u^h \right]
+  \frac{1}{h} D_q F \left[ u_x^h \mid_ {x-h} - \delta_{-h} u^h \right] = 0,
\end{aligned} \right. 
\end{equation}
where $w= (u^h_x)^2/2$ and $u^h_h = \partial u^h / \partial h$.
\end{Proposition}

\begin{proof}
Equations \eqref{js2}$_1$ and \eqref{js2}$_2$ are obtained by differentiating 
\eqref{HJ.approx} w.r.t. $x$ once {and} twice, respectively.
Then, \eqref{js2}$_3$ follows multiplying \eqref{js2}$_1$ by $u^h_x$ and taking into account \eqref{ww}.
Finally, differentiating \eqref{HJ.approx} w.r.t. $h$ we have
\begin{equation*} %\label{diffh}
(u^h_h)_t - D_p F \left[ \delta_h u^h_h + \frac{1}{h} ( u_x^h \mid_ {x+h} - \delta_h u^h) \right]
+ D_q F \left[ \delta_{-h} u^h_h + \frac{1}{h} ( u_x^h \mid_ {x-h} - \delta_{-h} u^h) \right] = 0,
\end{equation*}
which is \eqref{js2}$_4$.
\end{proof}
We show now some \textit{a priori bounds} which will be used in the proof of the main theorem.
\begin{Proposition}
Let $h > 0$. Then, for every $t \in [0, \infty)$
\begin{equation} \label{js4} 
\left\{ \begin{aligned}
&\| u_x^h (\cdot, t) \|_{L^{\infty} (\mathbb{T})} \leq \| (u_0)_x \|_{L^{\infty} (\mathbb{T})}, \\
& u_{xx}^h (\cdot, t) \leq \| (u_0)_{xx} \|_{L^{\infty} (\mathbb{T})}, \\
&\| \delta_{\pm h} u^h (\cdot, t) \|_{L^{\infty} (\mathbb{T})} \leq \| (u_0)_x \|_{L^{\infty} (\mathbb{T})}. 
\end{aligned} \right. 
\end{equation}
In particular, 
\begin{equation} \label{js3} 
\left\{ \begin{aligned}
& ( u_x^h \mid_{x + h} - \delta_{h} u^h )
\leq h \| (u_0)_{xx} \|_{L^{\infty} (\mathbb{T})}, \\
& - ( u_x^h \mid_{x - h} - \delta_{- h} u^h )
\leq h \| (u_0)_{xx} \|_{L^{\infty} (\mathbb{T})}, \\
&u_x^h - \delta_{h} u^h \geq - h \| (u_0)_{xx} \|_{L^{\infty} (\mathbb{T})}.
\end{aligned} \right. 
\end{equation}
\end{Proposition}

\begin{Remark}
We underline that in the proof of \eqref{js4}$_2$ and \eqref{js3} 
we use the convexity assumption on $F$. 
\end{Remark}

\begin{proof}
Let $t_1 \in (0, \infty)$, and choose $\overline{x} \in \mathbb{T}$ such that
\begin{equation*}
w (\overline{x}, t_1) = \max_{x \in \mathbb{T}} w (x, t_1).
\end{equation*}
Multiplying \eqref{js2}$_3$ by $\sigma^{h,\overline{x},t_1}$ and integrating, using Proposition \ref{formula}
\begin{equation*} \begin{split}
0 \geq \int_0^{t_1} \int_{\mathbb{T}}  \sigma^{h,\overline{x},t_1} L^h w \, dx \, dt
&= w (\overline{x},t_1) - \int_{\mathbb{T}} w (x,0) \sigma^{h,\overline{x},t_1} (x,0) \, dx \\
&= w (\overline{x},t_1) - \frac{1}{2} \int_{\mathbb{T}} \left( (u_0)_x \right)^2 (x,0) \sigma^{h,\overline{x},t_1} (x,0) \, dx,
\end{split} \end{equation*}
where the first inequality follows from 
the fact that $F$ is increasing in each variable.
Since $\sigma^{h,\overline{x},t_1} (\cdot,0)$ is a probability measure, \eqref{js4}$_1$ follows.

The second estimate is proven in a similar way.
Let $t_1 \in (0, \infty)$, and choose $\widehat{x} \in \mathbb{T}$ such that
\begin{equation*}
 u^h_{xx} (\widehat{x}, t_1)  = \max_{x \in \mathbb{T}} u^h_{xx} (x, t_1) .
\end{equation*}
Multiplying equation \eqref{js2}$_2$ by $\sigma^{h,\widehat{x},t_1}$, integrating, and using Proposition \ref{formula}
\begin{equation*} \begin{split}
0 \geq \int_0^{t_1} \int_{\mathbb{T}}  \sigma^{h,\widehat{x},t_1} L^h u^h_{xx} \, dx \, dt
&= u^h_{xx} (\widehat{x},t_1) - \int_{\mathbb{T}} u^h_{xx} (x,0) \sigma^{h,\widehat{x},t_1} (x,0) \, dx \\
&= u^h_{xx} (\widehat{x},t_1) - \int_{\mathbb{T}} (u_0)_{xx} \sigma^{h,\widehat{x},t_1} (x,0) \, dx,
\end{split} \end{equation*}
where the first inequality follows from the fact that $F$ is convex.
Last inequality implies \eqref{js4}$_2$.
Estimate \eqref{js4}$_3$ easily follows from \eqref{js4}$_1$.

Observe now that
\begin{equation*} \begin{split}
&u_x^h \mid_{x + h} - \delta_{h} u^h 
= u_x^h (x + h) - \frac{u^h (x + h) - u^h (x)}{h} \\
&=  u_x^h (x + h) - u_x^h (x + \tau h) 
= u_{xx}^h (x + \tau \eta h) (1 - \tau ) h ,
\end{split} \end{equation*}
for some $\tau, \eta \in (0,1)$, and this gives \eqref{js3}$_1$.
In a similar way one can prove \eqref{js3}$_2$ and \eqref{js3}$_3$.
\end{proof}
The next proposition gives an upper bound for $u^h_h$.
\begin{Proposition}
There exists a positive constant $C$ such that 
\begin{equation*}
 \max_{x \in \mathbb{T}} u^h_h ( x , t_1) \leq C t_1,
\end{equation*}
for every $h > 0$ and $t_1 \in (0, \infty)$.
\end{Proposition}

\begin{proof}
Let $t_1 \in (0, \infty)$ and choose $\overline{x}$ such that
\begin{equation*}
 u^h_h (\overline{x},t_1) = \max_{x \in \mathbb{T}} u^h_h ( x , t_1) . 
\end{equation*}
Then, multiplying equation \eqref{js2}$_4$ by $\sigma^{h,\overline{x},t_1}$, integrating, and using Proposition \ref{formula} 
\begin{equation*}
u^h_h (\overline{x},t_1) = \int_0^{t_1} \int_{\mathbb{T}} \left[ \frac{1}{h} D_p F \left[  u_x^h \mid_ {x+h} - \delta_h u^h \right] 
-  \frac{1}{h} D_q F \left[ u_x^h \mid_ {x-h} - \delta_{-h} u^h \right] \right] \sigma^{h, \overline{x},t_1} \, dx \, dt,
\end{equation*}
where we used the fact that $u^h_h ( \cdot ,0) \equiv 0$.
{Inequalities above}, together with \eqref{js4}$_3$, \eqref{js3}$_1$ and \eqref{js3}$_2$, 
{imply}
\begin{equation*}
 \frac{1}{h} D_p F \left[  u_x^h \mid_ {x+h} - \delta_h u^h \right] 
-  \frac{1}{h} D_q F \left[ u_x^h \mid_ {x-h} - \delta_{-h} u^h \right] \leq C, 
\end{equation*}
for some positive constant $C$ independent of $h$, so that the conclusion follows.
\end{proof}

\begin{Proposition}
There exists a positive constant $C$ such that 
\begin{equation*}
 \min_{x \in \mathbb{T}} u^h_h ( x , t_1) \geq - \frac{1}{\sqrt{h}} C ( 1 + t_1 ),
\end{equation*}
for every $h > 0$ and $t_1 \in (0, \infty)$.
\end{Proposition}

\begin{proof}
Let $t_1 \in (0, \infty)$ and choose $\overline{x}$ such that
\begin{equation*}
 u^h_h (\overline{x},t_1) = \min_{x \in \mathbb{T}} u^h_h ( x , t_1) . 
\end{equation*}
As in the previous proof, we have 
\begin{equation*}
u^h_h (\overline{x},t_1) = \int_0^{t_1} \int_{\mathbb{T}} \left[ \frac{1}{h} D_p F \left[  u_x^h \mid_ {x+h} - \delta_h u^h \right] 
-  \frac{1}{h} D_q F \left[ u_x^h \mid_ {x-h} - \delta_{-h} u^h \right] \right] \sigma^{h, \overline{x},t_1} \, dx \, dt.
\end{equation*}
Using Young's inequality and \eqref{js3}$_3$
\begin{equation} \begin{split} \label{vc}
& \frac{1}{h} D_p F \left[  u_x^h \mid_ {x+h} - \delta_h u^h \right] 
= D_p F ( \delta_h u^h_x ) + \frac{1}{h} D_p F ( u_x^h - \delta_h u^h ) \\
&\geq D_p F ( \delta_h u^h_x ) -  C
\geq - \frac{1}{2} \frac{D_p F}{\sqrt{h}} - \frac{\sqrt{h}}{2} ( D_p F ) ( \delta_h u^h_x )^2 - C.
\end{split} \end{equation}
In a similar way we obtain 
\begin{equation} \label{vb}
-  \frac{1}{h} D_q F \left[ u_x^h \mid_ {x-h} - \delta_{-h} u^h \right]
\geq - \frac{1}{2} \frac{D_q F}{\sqrt{h}} - \frac{\sqrt{h}}{2} ( D_q F ) ( \delta_{-h} u^h_x )^2 - C.
\end{equation}
Thus, adding relations \eqref{vc} and \eqref{vb}
\begin{equation} \label{evans.adj} \begin{split}
u^h_h (\overline{x},t_1) &\geq 
- \frac{1}{2 \sqrt{h}} \int_0^{t_1} \int_{\mathbb{T}} \left[  D_p F + D_q F \right] 
\sigma^{h, \overline{x},t_1} \, dx \, dt - 2 C t_1 \\
&- \frac{1}{\sqrt{h}} \frac{h}{2} \int_0^{t_1} \int_{\mathbb{T}} \left[
D_p F (\delta_h u^h_x)^2 + D_q F (\delta_{-h} u^h_x)^2 \right] 
\sigma^{h, \overline{x},t_1} \, dx \, dt
\geq - \frac{1}{\sqrt{h}} C ( 1 + t_1 ). 
\end{split} \end{equation}

\end{proof}
The next result is a direct consequence of the previous two propositions 
and implies Theorem~\ref{main}.
\begin{Proposition}
There exists a positive constant $C$ such that 
\begin{equation*}
\| u^h_h ( \cdot , t) \|_{L^{\infty} (\mathbb{T})} 
 \leq \frac{1}{\sqrt{h}} C ( 1 + t ),
\end{equation*}
for every $h > 0$ and $t \in (0, \infty)$.
\end{Proposition}

\begin{Remark} \label{convex1}
To prove \eqref{evans.adj} we used the new inequality
\begin{equation}\label{new.inequ1}
h \int_0^{t_1} \int_{\mathbb T} 
[D_p F (\delta_h u^h_x)^2 + D_q F (\delta_{-h} u^h_x)^2]
\sigma^{h, \overline{x},t_1} \, dx \, dt \le C,
\end{equation}
which can be easily derived by multiplying \eqref{js2}$_3$ by
$\sigma^{h, \overline{x},t_1}$ and integrating by parts.
If we choose $F$ as in \eqref{H.CL}, then \eqref{new.inequ1} reads as
\begin{equation}\label{new.inequ2}
h \int_0^{t_1} \int_{\mathbb T} 
[ (\delta_h u^h_x)^2 +(\delta_{-h} u^h_x)^2]
\sigma^{h, \overline{x},t_1} \, dx \, dt \le C,
\end{equation}
which is the analog of the new and important inequality
$$
\ep  \int_0^{t_1} \int_{\mathbb T}
|D^2 u^\ep|^2 \sigma^\ep \,dx\,dt \le C,
$$
which Evans derived in \cite{E2}.
Notice that \eqref{new.inequ1} and \eqref{new.inequ2} hold for general (non convex)
coercive Hamiltonians.
However, we do not know whether \eqref{new.inequ2} is still correct 
if we replace $\delta_h u^h_x$ by $u^h_{xx}$ or 
by $\dfrac{u^h_x-\delta_h u^h}{h}$.
That is one of the reasons why we have to require the convexity assumption
on $F$ in order to have \eqref{js3}$_3$ which we use, for instance, in proving \eqref{vc} and \eqref{vb}.
\end{Remark}

\begin{Remark} 
If $F$ is as in \eqref{H.CL}, and we assume further that
$H$ is uniformly convex, we can improve \eqref{new.inequ2}.
Indeed, let $\sigma^{h,\nu,t_1}$ be a solution of the adjoint equation
\begin{equation*}
\left\{ \begin{aligned}
- \sigma^{h,\nu,t_1}_t + \delta_{-h}( \sigma^{h,\nu,t_1} D_p F ) 
- \delta_{h}( \sigma^{h,\nu,t_1} D_q F ) 
&= 0, \hspace{.5cm} \quad \text{ in } \mathbb{T} \times [0, t_1 ), \\
\sigma^{h,\nu,t_1} &= \nu, \quad \quad \, \,  \text{ on } \mathbb{T} \times \{ t= t_1 \},
\end{aligned} \right. 
\end{equation*}
where $\nu$ is a probability measure on $\mathbb T$ with a smooth density.
Then, multiplying \eqref{js2}$_2$ by $\sigma^{h,\nu,t_1}$ and integrating by parts we have
\begin{equation}\label{new.inequ3}
\int_0^{t_1} \int_{\mathbb T} 
[ (\delta_h u^h_x)^2 +(\delta_{-h} u^h_x)^2]
\sigma^{h, \nu,t_1} \, dx \, dt \le C,
\end{equation}
for some $C=C(t_1,\nu)$. 
See \cite{E2,CGT1} for more applications of inequalities \eqref{new.inequ1}, \eqref{new.inequ2}
and \eqref{new.inequ3}. 
\end{Remark}
{In the next subsection we prove the $L^1$-error estimate.}

%%%%%%%%%%%%%%%%%%%%%%%%%%%%%%%%%%%%%%%%%%

\subsection{$L^1$-error estimates}\label{L1new}

In this subsection the numerical Hamiltonian is of the form
$$
F(p,q)=H \left( \dfrac{q-p}{2} \right)+\gamma(p+q).
$$
{Before proving Theorem \ref{main2}, we need two preliminary lemmas.} 
\begin{Lemma} \label{L1bd}
There exists $C>0$ such that
\begin{equation}\label{L1}
\int_{\mathbb T}( |\Delta u^h(x,t)| + |\Delta_h u^h(x,t)|) \, dx \le C, \quad
\text{for any } t>0.
\end{equation}
\end{Lemma}
\begin{proof}
By $\eqref{js4}_2$, we have
$$
\Delta u^h(x,t),\ \Delta_h u^h(x,t) \le \|\Delta u_0\|_{L^\infty(\mathbb T)} \le C.
$$
It is therefore easy to see that
\begin{align*}
&|\Delta u^h(x,t)|+|\Delta_h u^h(x,t)|
 = 2(\Delta u^h(x,t))^{+}+ 2(\Delta_h u^h(x,t))^{+} - \Delta u^h (x,t)
-\Delta_h u^h(x,t)
\\
 \le\ &C - \Delta u^h(x,t) {-}\Delta_h u^h(x,t).
\end{align*}
Integrate the above inequality over $\mathbb T$ to achieve
$$
\int_{\mathbb T} (|\Delta u^h(x,t)|+|\Delta_h u^h(x,t)|)\,dx
\le \int_{\mathbb T} ( C - \Delta u^h(x,t)- \Delta_h u^h(x,t))\,dx = C.
$$
\end{proof}
\begin{Remark}\label{L1sim}
By using the same {argument} of Lemma \ref{L1bd}, we can derive the following estimate
\begin{equation}\label{L1s}
\int_{\mathbb T} \dfrac{1}{h}(|u^h_x(x+h,t)-\delta_h u^h(x,t)|+|u^h_x(x-h,t)-\delta_{-h} u^h(x,t)|) \,dx \le C.
\end{equation}
\end{Remark}

Let us now recall the Adjoint equation with different choices of terminal data.
For each $\nu \in L^\infty(\mathbb T)$, we denote by $\sigma^{h,\nu,T}$ the solution of
\begin{equation} \label{L1ad} 
\left\{ \begin{aligned}
- \sigma^{h,\nu,T}_t + \delta_{-h}( \sigma^{h,\nu,T} D_p F ) 
- \delta_{h}( \sigma^{h,\nu,T} D_q F ) 
&= 0, \qquad \quad \quad \quad \quad \text{ in } \mathbb{T} \times [0, T ), \\
\sigma^{h,\nu,T} &=\nu, \qquad \quad \quad \quad \quad \text{ on } \mathbb{T} \times \{ t= T \}.
\end{aligned} \right. 
\end{equation}
By abuse of notation, we write $\sigma^\nu$ for $\sigma^{h,\nu,T}$.

\begin{Lemma}\label{L1ad.bd}
There exists $C=C(\|\nu\|_{L^\infty(\mathbb T)},T)$ such that
$$
\|\sigma^\nu\|_{L^\infty(\mathbb T \times [0,T])} \le C.
$$
\end{Lemma}
This Lemma is an analogous version of the Maximum principle for parabolic equations.
Notice that the convexity of $F$ and the uniform semiconcavity of $u^h$ are crucial here.
\begin{proof}
The idea of the proof is an application of the Maximum principle.
{By direct computations, thanks to \eqref{ww}}, \eqref{L1ad} reads
$$
-\sigma^\nu_t + (\delta_{-h}(D_p F) -\delta_h(D_qF)) \sigma^\nu 
+ \delta_{-h}(\sigma^\nu) D_p F|_{x-h} - \delta_h \sigma^\nu D_q F|_{x+h}=0.
$$
Note that $D_p F(p,q)=-\dfrac{1}{2} {H'} \big( \dfrac{q-p}{2} \big) + \gamma$ and
$D_q F(p,q)=\dfrac{1}{2} {H'} \big( \dfrac{q-p}{2} \big) + \gamma$.
By {the} Mean Value Theorem, there exists $s \in (0,1)$ such that
\begin{align*}
\delta_{-h}(D_p F) -\delta_h(D_qF)
&=-\dfrac{1}{2} \sum_{m\in \{-1,1\}}  {H'' \left( \dfrac{1}{2} (\delta_h u^h+\delta_{-h} u^h) \mid_{(x+msh,t)} \right) 
\dfrac{\delta_h u^h_x+\delta_{-h} u^h_x}{2}\mid_{(x+msh,t)}} \\
&\ge -K,
\end{align*}
where 
$$
K=\max_{|p| \le C_1} H''(p) \times \max_{x\in \mathbb T} (\Delta u^h(x))^+ \ge 0
$$
with $C_1=\|(u_0)_x\|_{L^\infty(\mathbb T)}$, which is the uniform bound
for $u^h_x$ as in $\eqref{js4}_1$.

{
Let $\beta(s)=\max_{x \in \mathbb T} |\sigma^\nu(x,s)|$ then 
by Maximum principle, we straightforwardly derive that
$$
\beta'(s)+K\beta(s) \ge 0, \quad \text{for } s \in (0,T),
$$
in the viscosity sense.
Thus, we easily get $\beta(s) \le e^{K(T-s)} \|\nu\|_{L^\infty(\mathbb T)}$,
which completes the proof.
}
\end{proof}

\begin{proof} [Proof of Theorem \ref{main2}]
As usual, we multiply $\eqref{js2}_4$ by $\sigma^\nu$ and integrate by parts to get
\begin{align*}
&\int_{\mathbb T} u_h^h(x,T)\nu(x)\,dx \\
&=\int_0^T \int_{\mathbb T} \dfrac{1}{h}( D_p F  (u^h_x(x+h,t)-\delta_h u^h(x,t))
-D_qF (u^h_x(x-h,t)-\delta_{-h} u^h(x,t)))\sigma^\nu(x,t)\,dx\,dt.
\end{align*}
Now, notice that
$$
\int_{\mathbb T} |u_h^h(x,T)|\, dx = \sup_{\nu \in L^\infty(\mathbb T), \| \nu \|_{L^\infty(\mathbb T)} \le 1} 
\int_{\mathbb T} u_h^h(x,T)\nu(x)\,dx.
$$
For $ \| \nu \|_{L^\infty(\mathbb T)} \le 1$, Lemma \ref{L1ad.bd} gives us that
\begin{equation}\label{L1f}
\|\sigma^\nu\|_{L^\infty(\mathbb T \times [0,T])} \le C.
\end{equation}
By using  \eqref{L1s}, \eqref{L1f}, we obtain
$$
\int_{\mathbb T} |u_h^h(x,T)|\, dx \le C.
$$
\end{proof}

%%%%%%%%%%%%%%%%%%%%%%%%%%%%%%%%%%%%%%%%%%%%

\subsection{Generalizations} \label{generalization}
Theorems \ref{main} and \ref{main2} can be generalized easily
to higher {dimensions as follows}.
We consider the following Hamilton--Jacobi equation
\begin{equation*}
\left\{ \begin{aligned}
u_t + H (Du) &= 0, \quad \quad \, \ \text{ in } \mathbb{T}^n \times (0,\infty), \vspace{.05in} \\
u &= u_0 , \quad \quad \text{ on } \mathbb{T}^n \times \{ t = 0 \},
\end{aligned} \right. 
\end{equation*}
where the Hamiltonian $H:\mathbb R^n \to \mathbb R$ is smooth, coercive, and convex,
and $u_0: \mathbb T^n \to \mathbb R$ is a given smooth function.

We define the numerical Hamiltonian $F$ to be given explicitly {as follows}
$$
F(p,q)=F(p_1,\ldots, p_n, q_1,\ldots, q_n)=H \left ( \dfrac{q-p}{2} \right )
+ \gamma ( p_1+\ldots+ p_n+q_1+\ldots+q_n),
$$
where $\gamma$ is a positive constant chosen as in \eqref{H.CL}.

The adjoint equation then is
\begin{equation*}
\left\{ \begin{aligned}
- \sigma^{h,\nu,T}_t + \delta_{-h}( \sigma^{h,\nu,T} D_p F ) 
- \delta_{h}( \sigma^{h,\nu,T} D_q F ) 
&= 0, \qquad \quad \quad \quad \quad \text{ in } \mathbb{T}^n \times [0, T ), \\
\sigma^{h,\nu,T} &=\nu, \qquad \quad \quad \quad \quad \text{ on } \mathbb{T}^n \times \{ t= T \},
\end{aligned} \right. 
\end{equation*}
where the terminal datum $\nu$ {can be chosen as 
a Dirac measure or as an $L^\infty$ function, in order to prove Theorem \ref{main}
or Theorem \ref{main2}, respectively.}

All the derivations in Sections \ref{Linf} and \ref{L1new} still hold straightforwardly.
Let us emphasize that the convexity of $H$ and the uniformly semiconcavity of $u^h$
are crucial in this approach.

%%%%%%%%%%%%%%%%%%%%%%%%%%%%%%%%%%%%%%%%%%%
\begin{subsection}{An additional estimate} \label{42}
Let us now choose $F$ as in \eqref{H.CL};
then equation \eqref{y} becomes
\begin{equation}\label{y.CL}
u^h_t + H \left( \dfrac{\delta_h u^h +\delta_{-h} u^h}{2} \right) = \gamma h \Delta_h u^h.
\end{equation}

We are able to get the following estimate
\begin{Lemma}\label{comp}
There exists $C>0$, independent of $h$ and $T$, such that
\begin{equation}\label{new.inequ4}
 h \int_0^{T} \int_{\mathbb T} 
 \Delta_h u^h (\delta_h u^h_x +\delta_{-h} u^h_x) \,dx\,dt \le C,
 \qquad \text{ for every }h,T > 0.
 \end{equation}
\end{Lemma}
\begin{proof}
Differentiate \eqref{y.CL} w.r.t. $x$, and  then multiply 
by $\delta_h u^h+\delta_{-h} u^h$, to get
\begin{equation}\label{y.x}
\begin{split}
&(\delta_h u^h+\delta_{-h} u^h) u^h_{xt} 
+ \dfrac{1}{2}H' \left( \dfrac{\delta_h u^h+\delta_{-h} u^h}{2} \right)
(\delta_h u^h+\delta_{-h} u^h) (\delta_h u^h_x+\delta_{-h} u^h_x) \\
&\hspace{7cm}= \gamma h \Delta_h u^h_x (\delta_h u^h+\delta_{-h} u^h).
\end{split}\end{equation}
Choose $G$ such that $G'(s) = 2H'(s)s$ for $s\in \mathbb R$ then
\begin{equation} \label{su} \begin{split}
&\int_0^{T} \int_{\mathbb T} 
\dfrac{1}{2}H' \left(\dfrac{\delta_h u^h+\delta_{-h} u^h}{2} \right)
(\delta_h u^h+\delta_{-h} u^h) (\delta_h u^h_x+\delta_{-h} u^h_x)\,dx\,dt\\
=&\int_0^{T} \int_{\mathbb T} 
G' \left( \dfrac{\delta_h u^h+\delta_{-h} u^h}{2} \right) 
\left( \dfrac{\delta_h u^h_x+\delta_{-h} u^h_x}{2} \right) \,dx\,dt=0.
\end{split} \end{equation}
Integrating the first term in the left hand side of \eqref{y.x}, we have
\begin{align*}
L_1&=\int_0^{T} \int_{\mathbb T} (\delta_h u^h+\delta_{-h} u^h) u^h_{xt} \,dx\,dt\\
&= \left[ \int_{\mathbb T} (\delta_h u^h+\delta_{-h} u^h) u^h_{x} \,dx \right]_{t=0}^{t=T}
+\int_0^{T} \int_{\mathbb T} (\delta_h u^h_{xt}+\delta_{-h} u^h_{xt}) u^h \,dx\,dt\\
&= \left[ \int_{\mathbb T} (\delta_h u^h+\delta_{-h} u^h) u^h_{x} \,dx \right]_{t=0}^{t=T}
-\int_0^{T} \int_{\mathbb T} (\delta_h u^h+\delta_{-h} u^h) u^h_{xt} \,dx\,dt\\
&= \left[ \int_{\mathbb T} (\delta_h u^h+\delta_{-h} u^h) u^h_{x} \,dx \right]_{t=0}^{t=T} - L_1,
\end{align*}
and therefore, using \eqref{js4}$_1$ and \eqref{js4}$_3$,
\begin{equation} \label{giu}
L_1=\dfrac{1}{2}  \left[ \int_{\mathbb T} 
(\delta_h u^h+\delta_{-h} u^h) u^h_{x} \,dx \right]_{t=0}^{t=T} \ge -C.
\end{equation}
Integrating \eqref{y.x} and taking into account \eqref{su} and \eqref{giu}
\begin{align*}
-C &\le \int_0^{T} \int_{\mathbb T} \gamma h \Delta_h u^h_x
(\delta_h u^h + \delta_{-h} u^h) \,dx \,dt
=-\int_0^{T} \int_{\mathbb T} \gamma h \Delta_h u^h
(\delta_h u^h_x +\delta_{-h} u^h_x) \,dx\,dt,
\end{align*}
from which \eqref{new.inequ4} follows.
\end{proof}
\begin{Remark}
Inequality \eqref{new.inequ4} is the analog of the following one
\begin{equation}\label{ccc}
\ep \int_0^{T} \int_{\mathbb T} |u^\ep_{xx}|^2 \,dx\,dt \le C
\end{equation}
if we consider the usual regularized equation
$$
u^\ep_t + H(u^\ep_x) = \ep u^\ep_{xx}
$$
and the space dimension is $1$.

Note that \eqref{ccc} was used in the context of Compensated Compactness
for $1$-dimensional conservation laws (see \cite{Tar, E4}).
We hope to revisit \eqref{new.inequ4} and \eqref{ccc} in the future
to study the shock structure of the solutions of the numerical scheme.

\end{Remark}
\end{subsection}

\end{section}

%%%%%%%%%%%%%%%%%%%%%%%%%%%%%%%%%%%%%%%%%%%%%%

\begin{section}{A special case: $H (p) = p^2/2$}\label{special.case}
 
We consider in this section the special case
\begin{equation*}
H (p) = \frac{ p^2}{2}.
\end{equation*}
Hence, we will study the Hamilton-Jacobi equation 
\begin{equation} \label{firsteq1}
\left\{ \begin{aligned}
u_t + \frac{u_x^2}{2} &= 0, \quad \quad \, \, \, \text{ in } \mathbb{T} \times (0,\infty), \vspace{.05in} \\
u &= u_0 , \quad \quad \text{ in } \mathbb{T} \times \{ t = 0 \}.
\end{aligned} \right. 
\end{equation}
We choose $F: \mathbb{R} \times \mathbb{R} \to \mathbb{R}$ defined as
\begin{equation*}
F (p,q):= \frac{(p^+)^2}{2} + \frac{(q^+)^2}{2},
\end{equation*}
where we used the notation
\begin{equation*}
 a^+:= \max \{a,0\}, \quad \quad a^-:= \min \{a,0\}, \quad \quad a \in \mathbb{R}.
\end{equation*}
Notice that in this case properties (F1)--(F3) are satisfied. In particular
\begin{equation*}
F (-p,p) =  \frac{((-p)^+)^2}{2} + \frac{(p^+)^2}{2} = \frac{(p^-)^2}{2} + \frac{(p^+)^2}{2} = \frac{p^2}{2} = H (p),
\end{equation*}
so that (F3) holds.
For every $h > 0$, we are then lead to study the following approximation of equation \eqref{firsteq1}:
\begin{equation} \label{apprh}
\left\{ \begin{aligned}
u^h_t +  \frac{\left[ ( - \delta_h u^h )^+ \right]^2}{2}
+ \frac{\left[ ( \delta_{-h} u^h )^+ \right]^2}{2} &= 0
, \quad \quad \, \, \, \text{ in } \mathbb{T} \times (0,\infty), \vspace{.05in} \\
u^h &= u_0 , \quad \quad \text{ in } \mathbb{T} \times \{ t = 0 \},
\end{aligned} \right. 
\end{equation}
or equivalently, 
\begin{equation*} %\label{apprh}
\left\{ \begin{aligned}
u^h_t +  \frac{\left[ ( \delta_h u^h )^- \right]^2}{2}
+ \frac{\left[ ( \delta_{-h} u^h )^+ \right]^2}{2} &=  0,
 \quad \quad \, \, \, \text{ in } \mathbb{T} \times (0,\infty), \vspace{.05in} \\
u^h &= u_0 , \quad \quad \text{ in } \mathbb{T} \times \{ t = 0 \},
\end{aligned} \right. 
\end{equation*}
where we used the fact that $( - \delta_h u^h )^+ = - ( \delta_h u^h )^-$.
The linear operator correspondent to \eqref{apprh} is given by
\begin{equation*}
 v \longmapsto L^h v :=
v_t  + ( \delta_h u^h )^- ( \delta_h v )
+ ( \delta_{-h} u^h )^+ ( \delta_{-h} v ).
\end{equation*}
Observe that, although the function $F$ just defined is not of class $C^2$, we have $F \in C^{1,1}$. 
Then, we can approximate $F$ with a sequence of smooth functions satisfying (F1)--(F3)
with equibounded Hessian (for instance by convolution). 
Thus, since all the constants appearing in the previous section just depend on the bounds on $D F$,
we can pass to the limit and still obtain Theorem~\ref{main}.

\end{section}

%%%%%%%%%%%%%%%%%%%%%%%%%%%%%%%%%%%%%%%%%%%%%%%%

\begin{section}{Appendix}

In this section we study the properties of the solution $u^h$ of equation
\begin{equation} \label{HJ.approx2}
\left\{ \begin{aligned}
u^h_t + F \left( - \delta_h u^h  , \delta_{-h} u^h \right) &= 0, 
\quad \quad \, \, \, \text{ in } \mathbb{T} \times (0,\infty), \vspace{.05in} \\
u^h &= u_0 , \quad \quad \text{ on } \mathbb{T} \times \{ t = 0 \},
\end{aligned} \right. 
\qquad \qquad h > 0.
\end{equation}

\begin{proof}[Proof of Proposition \ref{ODE1}]

\mbox{ }

\vspace{4pt} 
\noindent \textbf{Step 1: local existence and uniqueness}. 
Consider the following ODE in the Banach space $C (\mathbb{T})$:
\begin{equation} %
\left\{ \begin{aligned}
\dot{z}^h (t) &= G^h(z^h (t))  \quad \quad t \in (0,\infty) \vspace{.05in} \\
z^h (0) &= u_0  
\end{aligned} \right. 
\end{equation}
where $G^h: C(\mathbb{T}) \to C(\mathbb{T})$ is given by
\begin{equation} \label{OD}
G^h (z):= - F \left( - \delta_h z  , \delta_{-h} z \right).
\end{equation}
Here with the dot we denoted the derivative of the function
$[0, \infty) \ni t \mapsto z^h (t) \in C (\mathbb{T})$. 
Since $G^h$ is locally Lipschitz continuous, there exists $\delta > 0$
and a unique function $z^h \in C^1 ([0,\delta); C (\mathbb{T}))$ satisfying \eqref{OD}
for $t \in [0,\delta)$.
In particular, from the fact that $z^h \in C^1 ([0,\delta); C (\mathbb{T}))$
it follows that $(x,t) \mapsto z^h (x,t) \in C (\mathbb{T} \times [0,\delta))$ 
and $z^h (x, \cdot) \in C^1 ( [0,\delta))$ for every $x \in \mathbb{T}$.
Thus, $z^h$ is a solution to \eqref{HJ.approx2}.
On the other hand, every solution of \eqref{HJ.approx2} has to satisfy 
\eqref{OD} as well.
This shows local existence and uniqueness of $u^h$.

\vspace{4pt} 
\noindent \textbf{Step 2: global existence and uniqueness}. 
We claim that 
\begin{equation} \label{gen.prop1}
\| u^h (\cdot, t) \|_{L^{\infty} (\mathbb{T})} \leq \| u_0 \|_{L^{\infty} (\mathbb{T})} + | F (0,0) | \, t,
\qquad \text{ for every }t \in (0, \infty).
\end{equation}
To prove the claim fix $t_1 > 0$, choose any constant $c_1 <  F (0,0) $, and set $v^h:= u^h + c_1 t$.
Let $(\overline{x}, \overline{t}) \in \mathbb{T} \times [0,t_1]$ be such that
\begin{equation} \label{max}
v^h ( \overline{x} , \overline{t}) = \max_{ ( x , t )\in \mathbb{T} \times [0,t_1]} v^h ( x , t).
\end{equation}
Assume that $\overline{t} \in (0, t_1]$. Then,
\begin{equation*} \begin{split}
&v^h_t (\overline{x}, \overline{t}) = u^h_t (\overline{x}, \overline{t}) + c_1 
= - F \left( - \delta_h u^h (\overline{x}, \overline{t}) , \delta_{-h} u^h (\overline{x}, \overline{t}) \right) + c_1 \\
&= - F \left( - \delta_h v^h  (\overline{x}, \overline{t}), \delta_{-h} v^h (\overline{x}, \overline{t}) \right) + c_1 
\leq - F (0,0) + c_1 < 0, 
\end{split} \end{equation*}
which is not possible by \eqref{max}.
This implies $\overline{t} = 0$.
Thus, we conclude by \eqref{max} that
\begin{equation*} 
\max_{ x \in \mathbb{T}} u^h ( x , t ) - F (0,0)  \, t
 \leq \max_{ x \in \mathbb{T}} u_0 (x),
 \qquad \text{ for every } t \in [0,t_1],
\end{equation*}
so that
\begin{equation*} 
\max_{ x \in \mathbb{T}} u^h ( x , t ) \leq \max_{ x \in \mathbb{T}} u_0 (x) + | F (0,0) | \, t,
 \qquad \text{ for every } t \in [0,t_1].
\end{equation*}
In the same way we can show that
\begin{equation*} 
\min_{ x \in \mathbb{T}} u^h ( x , t ) \geq \min_{ x \in \mathbb{T}} u_0 (x) - | F (0,0) | \, t,
 \qquad \text{ for every } t \in [0,t_1].
\end{equation*}
This shows \eqref{gen.prop1} and, in turn, global existence and uniqueness.

\vspace{4pt} 
\noindent \textbf{Step 3: smoothness}. 
Consider the following equation
\begin{equation} \label{ODE2}
\left\{ \begin{aligned}
\dot{v}^h (t) &= P^h(t,v^h (t))  \quad \quad t \in (0,\infty), \vspace{.05in} \\
v^h (0) &= (u_0)_x, 
\end{aligned} \right. 
\end{equation}
where $P^h : (0, \infty) \times C (\mathbb{T}) \to C (\mathbb{T})$ is defined as
the formal linearization of $G^h$:
\begin{equation*} %\label{y2}
P^h ( t, w ) =  D_p F \mid_{( - \delta_h u^h , \delta_{-h} u^h )} \delta_h w 
- D_q F \mid_{( - \delta_h u^h , \delta_{-h} u^h )} \delta_{-h} w .
\end{equation*}
Since $D F$ is continuous, $P^h$ is continuous and $P^h (t,\cdot)$ is linear.
Then, there exists a unique global solution to \eqref{ODE2}.
By repeating what was done in the previous step, we have that
$(x,t) \mapsto v^h (x,t) \in C (\mathbb{T} \times [0,\infty))$ 
and $v^h (x, \cdot) \in C^1 ( [0,\infty))$ for every $x \in \mathbb{T}$.
We claim that $v^h = u^h_x $.

To show this observe that, for every $y \in \mathbb{R} \setminus \{ 0 \}$,
$\delta_y u^h \in C^1 ((0, \infty); C(\mathbb{T}))$ is the unique solution of the equation
\begin{equation*} %\label{ODE2}
\left\{ \begin{aligned}
\dot{w} (t) & = R^h (t, w (t))
  \quad \quad t \in (0,\infty), \vspace{.05in} \\
w (0) &= \delta_y u_0, 
\end{aligned} \right. 
\end{equation*}
where $R^h$ is given by
\begin{equation*}
R^h (t,z) := D_p F \mid_{\xi} \delta_h z - D_q F \mid_{\xi} \delta_{-h} z,
\end{equation*}
with
\begin{equation*}
\xi:= \left( -\theta \delta_h u^h (\cdot) - (1 - \theta) \delta_h u^h (\cdot + y) \, , \, 
\theta \delta_{-h} u^h (\cdot) + (1 - \theta) \delta_{-h} u^h (\cdot + y) \right), 
\end{equation*}
for some $\theta = \theta (t,y) \in (0,1)$.
Also, we have
\begin{equation*}
\| P^h ( t, w_2 ) - P^h ( t, w_1 ) \|_{C (\mathbb{T})} \leq  C_1 \| w_2 - w_1  \|_{C (\mathbb{T})}, 
\qquad C_1 = C_1 (t, h),
\end{equation*}
and
\begin{equation*}
\| P^h ( t, v^h (t)) - R^h ( t, v^h (t) ) \|_{C (\mathbb{T})} \leq \varphi^{h,y} (t),
\end{equation*}
where
\begin{equation*} \begin{split}
 \varphi^{h,y} (t) &:= \| \left[ D_p F \mid_{\xi}
 - D_p F \mid_{( - \delta_h u^h , \delta_{-h} u^h )} \right] \delta_h v^h (t) \|_{C (\mathbb{T})} \\
&\hspace{.5cm}+ \| \left[ D_q F \mid_{\xi} 
- D_q F \mid_{( - \delta_h u^h , \delta_{-h} u^h )} \right] \delta_{-h} v^h (t) \|_{C (\mathbb{T})}
\end{split} \end{equation*}
satisfies
\begin{equation*}
\lim_{y \to 0} \sup_{t \in [0,T]} \varphi^{h,y} (t) = 0, \qquad \text{ for every }T > 0 
\text{ and } h >0.
\end{equation*}
Using the version of Gronwall's Inequality stated at the end of the section we have
\begin{equation*}
\| \delta_y u^h (t) - v^h (t) \|_{C (\mathbb{T})} 
\leq e^{C_1 t} \| \delta_y u_0 - (u_0)_x \|_{C (\mathbb{T})} 
+ e^{C_1 t} \int_0^t e^{-C_1 s} \varphi^{h,y} (s) ds,
\end{equation*}
for every $t \in (0, \infty)$.
From this, we conclude that $(u^h)_x (\cdot, t)= v^h (\cdot, t)$ for every $t \in [0, \infty)$
and thus $u^h (\cdot, t) \in C^1 (\mathbb{T})$.

In a similar way, one can show the part of the statement concerning $u^h_x$ and $u^h_{xx}$.
\end{proof}
We conclude by stating the version of Gronwall's inequality which was used in the previous proof.
\begin{Lemma}[Gronwall's inequality]
Let $X$ be a Banach space and $U \subset X$ an open set in $X$.
Let $f,g: [a,b] \times X \to X$ be continuous functions and let $y,z: [a,b] \to U$
satisfy the initial value problems
\begin{equation*} 
\left\{ \begin{aligned}
\dot{y} (t) &= f (t,y (t))  \quad \quad t \in (a,b), \vspace{.05in} \\
y (a) &= y_0, 
\end{aligned} \right.
\quad \quad \quad 
 \left\{ \begin{aligned}
\dot{z} (t) &= g (t,z (t))  \quad \quad t \in (a,b), \vspace{.05in} \\
z (a) &= z_0. 
\end{aligned} \right.
\end{equation*}
Also assume there is a constant $C \geq 0$ so that
\begin{equation*}
\| g (t,x_2) - g(t,x_1) \| \leq C \| x_2 - x_1 \|
\end{equation*}
and a continuous function $\varphi: [a,b] \to [0, \infty)$ so that
\begin{equation*}
\| f (t,y (t)) -  g (t,y (t)) \| \leq \varphi (t).
\end{equation*}
Then for $t \in [a,b]$ 
\begin{equation*}
\| y (t) - z (t) \| \leq e^{C |t-a|} \| y_0 - z_ 0 \|
+  e^{C |t-a|} \int_a^t e^{-C (s- a)} \varphi (s) \, ds.
\end{equation*}
\end{Lemma}

\end{section}

%%%%%%%%%%%%%%%%%%%%%%%%%%%%%%%%%%%%%%%%%%%%%%%%%

\end{document}